\documentclass{article}

\RequirePackage{amsmath,amsfonts,amssymb}
\RequirePackage{mathtools}
\RequirePackage{graphicx}
\RequirePackage{pstricks}
\RequirePackage{epic,eepic}
\RequirePackage{color}
\RequirePackage{multicol}
\RequirePackage{longtable}
\RequirePackage{ifthen}
\RequirePackage{calc}
\RequirePackage[all,cmtip]{xy}

\DeclareMathOperator{\Ext}{Ext}

\DeclareMathOperator{\Hom}{Hom}

\DeclareMathOperator{\Aut}{Aut}

\newcommand{\gen}[1]{\left\langle#1\right\rangle} 

\newcommand{\ds}{\displaystyle}

\newcommand{\zz}{\mathbb{Z}}
\newcommand{\qq}{\mathbb{Q}}
\newcommand{\rr}{\mathbb{R}}

\newtheorem{theorem}{Theorem}
\newtheorem{proposition}{Proposition}
\newtheorem{corollary}{Corollary}
\newenvironment{proof}{\medskip\noindent\textit{Proof: \/}}
               {\begin{flushright}\rule{2mm}{2mm}\end{flushright}\par\medskip}
\newenvironment{remark}{\medskip\noindent\textbf{Remark: }}
               {\par\medskip}

\begin{document}

\pagestyle{myheadings}
\markright{Cohomology ring of group surfaces}

\title{Diagonal approximation and the cohomology ring\\ of the fundamental groups of surfaces} 
\date{}
\author{S\'ergio Tadao Martins\footnote{The author has received financial support from FAPESP, process number 2013/07510-4} \and Daciberg Lima Gon\c calves\footnote{Corresponding author: Dept. de Matem\'atica --- IME-USP, Caixa Postal 66.281 --- CEP 05314--970, S\~ao Paulo --- SP, Brasil; FAX: 55--11--30916183, email address: \texttt{dlgoncal@ime.usp.br}}}
\maketitle

\begin{abstract}
We construct finite free resolutions of $\zz$ over $\zz\pi$, where
$\pi$ is the fundamental group of a surface distinct from $S^2$ and
$\rr P^2$, and define diagonal approximations for these
resolutions. We then proceed to give some possible applications that
come with the knowledge of those maps.


\end{abstract}

\section{Introduction}

The closed surfaces other than $S^2$ and $\rr P^2$ are $K(\pi,1)$
spaces, which means that their cohomology rings coincide with the
cohomology rings of their fundamental groups. There are many
situations where one may be interested in the cohomology groups of the
surfaces. See, for instance,~\cite{DacibergEdson97}. Given that
interest, in this paper we define finite free resolutions for the
fundamental groups of the $K(\pi,1)$ surfaces and partial diagonal
approximations for these resolutions, which allow us to compute the
cohomology rings $H^*(\pi,M)$ for any coefficient $M$ in an efficient
way.

\bigskip
Let us briefly recall how to define and compute (at least in theory)
not only the cohomology groups $H^*(\pi,M)$ for a given group $\pi$
with coefficients in the $\zz\pi$-module $M$, but also how to
determine the multiplicative structure given by the cup product. More
details about the definitions can be found in~\cite{Brown82}.

If $M$ is a (left) $\zz\pi$-module, the cohomology group $H^n(\pi,M)$
is defined by $H^n(\pi,M) = \Ext_{\zz\pi}^n(\zz,M)$ for $n\ge 0$,
where $\zz$ is the trivial $\zz\pi$-module. Hence one way to compute
the groups $H^n(\pi,M)$ is this: first, we find a projective
resolution $P$ over the ring $\zz\pi$ of the trivial $\zz\pi$-module
$\zz$. Then we apply the functor
$\Hom_{\zz\pi}(\underline{\phantom{M}},M)$ to the chain complex $P$
and the cohomology groups $H^n(\pi,M)$ are the cohomology groups of
the chain complex $\Hom_{\zz\pi}(P,M)$. Also, if
\[
\xymatrix{
\cdots \ar[r] & P_n \ar[r]^-{d_n} & P_{n-1} \ar[r] & \cdots \ar[r] & P_0 \ar[r]^-{\varepsilon} & \zz \ar[r] & 0 \\
}
\]
is a projective resolution of $\zz$, then
\[
\xymatrix{
\cdots \ar[r] & (P\otimes P)_n \ar[r]^-{\partial_n} & (P\otimes P)_{n-1} \ar[r] & \cdots \ar[r] & (P\otimes P)_0 \ar[r]^-{\varepsilon\otimes\varepsilon} & \zz\otimes\zz \cong \zz \ar[r] & 0 \\
}
\]
is also a projective resolution, and there is a map of chain complexes
$\Delta\colon P\to (P\otimes P)$ such that
\[
\xymatrix{
\cdots \ar[r] & P_n \ar[r]^-{d_n} \ar[d]^-{\Delta_n} & P_{n-1} \ar[r] \ar[d]^-{\Delta_{n-1}} & \cdots \ar[r] & P_0 \ar[r]^-{\varepsilon} \ar[d]^-{\Delta_0} & \zz \ar[r]\ar@{=}[d] & 0 \\
\cdots \ar[r] & (P\otimes P)_n \ar[r]^-{\partial_n} & (P\otimes P)_{n-1} \ar[r] & \cdots \ar[r] & (P\otimes P)_0 \ar[r]^-{\varepsilon\otimes\varepsilon} & \zz \ar[r] & 0 \\
}
\]
is commutative. The map $\Delta$ is called a {\em diagonal
  approximation\/} for the resolution $P$ and is used to define the
cup product
\[
H^p(\pi,M) \otimes H^q(\pi,N) \stackrel{\smile}{\to} H^{p+q}(\pi, M\otimes N)
\]
in the following way: let $u\in \Hom_{\zz\pi}(P_p,M)$, $v\in
\Hom_{\zz\pi}(P_q,N)$ and let $\alpha\in H^p(G,M)$ and $\beta\in
H^q(\pi,N)$ be the classes of the homomorphisms $u$ and $v$,
respectively. The {\em cup produt\/} of $\alpha$ and $\beta$ is
defined by
\begin{equation}\label{eq:alphabeta}
(\alpha\smile \beta) = [(u\otimes v)\circ \Delta] \in
H^{p+q}(\pi,M\otimes N),
\end{equation}
where $[\phantom{M}]$ denotes the cohomology class. Given that
definition, we can be a little more specific: the product
$(\alpha\smile\beta)$ in~(\ref{eq:alphabeta}) is the cohomology class
of the map $(u\otimes v)\circ \Delta_{pq}$, where $\Delta_{pq}\colon
P_{p+q}\to P_p\otimes P_q$ is the composition of $\Delta_{p+q} \colon
P_{p+q}\to (P\otimes P)_{p+q}$ with the projection $\pi_{pq}\colon
(P\otimes P)_{p+q}\to (P_p\otimes P_q)$.

Thus, the computation of the cohomology groups $H^*(\pi,M)$ together
with the multiplicative structure given by the cup product can be
accomplished if we manage to find the free resolution $P$ and the
diagonal approximation $\Delta$ for the resolution $P$. This is no
easy task in general. In~\cite{TomodaPeter2008}, we find the following
two propositions, which can help us determine $\Delta$ and were used
by the authors to compute the cohomology of certain $4$-periodic
groups:

\begin{proposition}\label{proposition:tomodacontraction}
For a group $\pi$, let
\[
\xymatrix{
  \cdots \ar[r] & C_n \ar[r] & \cdots \ar[r] & C_1 \ar[r] & C_0 \ar[r]^-{\varepsilon} \ar[r] & \zz \ar[r] & 0
}
\]
be a finitely generated free resolution of $\zz$ over $\zz \pi$, that
is, each $C_n$ is finitely generated as a $\zz \pi$-module. If $s$ is a
contracting homotopy for the resolution $C$, then a contracting
homotopy $\tilde{s}$ for the free resolution $C\otimes C$ of $\zz$
over $\zz \pi$ is given by
\begin{align*}
\tilde{s}_{-1}\colon &\zz \to C_0\otimes C_0 \\
&\tilde{s}_{-1}(1) = s_{-1}(1)\otimes s_{-1}(1), \\
\tilde{s}_n\colon &(C\otimes C)_n \to (C\otimes C)_{n+1} \\
& \tilde{s}_n(u_i\otimes v_{n-i}) = s_i(u_i)\otimes v_{n-i} + s_{-1}\varepsilon(u_i)\otimes s_{n-i}(v_{n-i}), \quad\text{if $n\ge 0$},
\end{align*}
where $s_{-1}\varepsilon\colon C_0\to C_0$ is extended to
$s_{-1}\varepsilon = \{(s_{-1}\varepsilon)_n \colon C_n \to C_n\}$ in
such a way that $(s_{-1}\varepsilon)_n = 0$ for $n\ge 1$.
\end{proposition}

\begin{proposition}\label{proposition:tomodadiagonal}
For a group $\pi$, let
\[
\xymatrix{
  \cdots \ar[r] & C_n \ar[r]^-{d_n} & \cdots \ar[r] & C_1 \ar[r]^-{d_1} & C_0 \ar[r]^-{\varepsilon} \ar[r] & \zz \ar[r] & 0
}
\]
be a finitely generated free resolution of $\zz$ over $\zz \pi$ (i.e.,
each $C_n$ is a finitely generated free $\zz \pi$-module), and let $s$
be a contracting homotopy for this resolution $C$. If $\tilde{s}$ is
the contracting homotopy for the resolution $C\otimes C$ given by
Proposition~\ref{proposition:tomodacontraction}, then a diagonal
approximation $\Delta\colon C\to C\otimes C$ can be defined in the
following way: for each $n\ge 0$, the map $\Delta_n\colon C_n \to
(C\otimes C)_n$ is given in each generator $\rho$ of $C_n$ by
\[
\begin{array}{l}
\Delta_0 = s_{-1}\varepsilon\otimes s_{-1}\varepsilon,\\
\Delta_n(\rho) = \tilde{s}_{n-1}\Delta_{n-1}d_n(\rho), \quad\text{if $n\ge 1$}.
\end{array}
\]
\end{proposition}

Essentially, the above propositions tell us that if we have a finitely
generated free resolution $P$ of $\zz$ over $\zz\pi$, then a diagonal
approximation $\Delta\colon P\to (P\otimes P)$ can be calculated if we
have a contracting homotopy for the resolution $P$. This is the
approach we will use to determine partial diagonal approximations for
the resolutions of the fundamental groups of the $K(\pi,1)$
surfaces. The word ``partial'' has the following meaning: since we are
dealing with $K(\pi,1)$ surfaces, the finitely generated free
resolution $P$ is given by the augmented cellular chain complex of the
universal cover $\rr^2$ of the surface as a free $\pi$-complex. As the
comohology groups $H^n(\pi,M)$ are then trivial for $n\ge 3$, the
observation after~(\ref{eq:alphabeta}) tells that we only need the
maps $\Delta_0\colon P_0\to (P_0\otimes P_0)$, $\Delta_{01}\colon P_1
\to (P_0\otimes P_1)$, $\Delta_{02}\colon P_2 \to (P_0\otimes P_2)$
and $\Delta_{11}\colon P_2 \to (P_1\otimes P_1)$ to compute the
meaningful products $H^p(\pi,M) \otimes H^q(\pi,N)
\stackrel{\smile}{\to} H^{p+q}(\pi, M\otimes N)$.

In Section 2 we'll use the results above to determine diagonal
approximations for the fundamental groups of the $K(\pi,1)$ surfaces,
and in Section 3 we make some comments on applications.

\section{Free resolutions and diagonal approximations}

Let $M_n$ be the orientable surface of genus $n\ge 1$, with
fundamental group $G$ given by the presentation
\[
G = \pi_1(M_n) = \gen{a_1,b_1,a_2,b_2,\ldots,a_n,b_n \mid
  [a_1,b_1][a_2,b_2]\cdot\cdots\cdot [a_n,b_n]}.
\]
We also let $p_i = [a_i,b_i] = a_ib_ia_i^{-1}b_i^{-1}$ and $p =
p_1p_2\cdot\cdots\cdot p_n$.

\begin{proposition}[Free resolution for the orientable case]\label{proposition:orientableresolution}
A free resolution of $\zz$ over $\zz G$ is given by
\begin{equation}\label{eq:orientableresolution}
\xymatrix{
0 \ar[r] & P_2 \ar[r]^-{d_2} & P_1 \ar[r]^-{d_1} & P_0 \ar[r]^-\varepsilon & \zz \ar[r] & 0,
}
\end{equation}
where
\begin{align*}
P_0 &= \gen{x} \cong \zz G,\\
P_1 &= \gen{y_1,\ldots,y_n,z_1,\ldots,z_n} \cong \zz G^{2n},\\
P_2 &= \gen{w} \cong \zz G,
\end{align*}
and the maps $\varepsilon$, $d_1$ and $d_2$ are defined by
\begin{align*}
\varepsilon(x) &= 1, \\
d_1(y_i) &= (a_i-1)x, \\
d_1(z_i) &= (b_i-1)x, \\
d_2(w) &= \ds\sum_{i=1}^n\left(\frac{\partial p}{\partial a_i}y_i + \frac{\partial p}{\partial b_i}z_i\right),
\end{align*}
where the partial derivatives are the Fox derivatives.
\end{proposition}

\begin{proof}
Since the surface $M_n$ is a $K(G,1)$ space, its universal cover
$\rr^2$ is a free and contractible $G$-complex. Hence its augmented
cellular chain complex, which is proved to be
exactly~\ref{eq:orientableresolution} in~\cite{Fadell83}, is a free
resolution of $\zz$ over $\zz G$.
\end{proof}

We now proceed to construct a diagonal approximation $\Delta\colon
P\to (P\otimes P)$ for the resolution $P$ given by the above
proposition.

\begin{theorem}[Diagonal approximation for the orientable case]\label{theorem:orientablediagonal}
Let $P$ be the free resolution of $\zz$ over $\zz G$ given by
Proposition~\ref{proposition:orientableresolution}. A diagonal
approximation $\Delta\colon P\to (P\otimes P)$ is partially given by
{\allowdisplaybreaks
\begin{align*}
\Delta_0\colon & P_0 \to (P\otimes P)_0 \\
&\Delta_0(x) = x\otimes x,\\
\Delta_1\colon & P_1 \to (P\otimes P)_1 \\
&\Delta_1(y_i) = y_i\otimes a_ix + x\otimes y_i, \\
&\Delta_1(z_i) = z_i\otimes b_ix + x\otimes z_i, \\
\Delta_{02}\colon & P_2 \to (P_0\otimes P_2) \\
&\Delta_{02}(w) = x\otimes w, \\
\Delta_{11}\colon & P_2 \to (P_1\otimes P_1) \\
&\Delta_{11}(w) = \sum_{i=1}^n\biggl[\sum_{j=1}^{i-1} (p_1\cdots p_{j-1})(1-a_jb_ja_j^{-1})y_j\otimes(p_1\cdots p_{i-1})y_i +\\
&\phantom{\Delta_{11}(w)=\sum\biggl[} + \sum_{j=1}^{i-1} (p_1\cdots p_{j-1})a_j(1-b_ja_j^{-1}b_j^{-1})z_j\otimes (p_1\cdots p_{i-1})y_i\biggr] -\\
&\phantom{\Delta_{11}(w)=} - \sum_{i=1}^{n-1}\biggl[\sum_{j=1}^{i-1}\biggl((p_1\cdots p_{j-1})(1-a_jb_ja_j^{-1})y_j\otimes (p_1\cdots p_{i-1})a_ib_ia_i^{-1}y_i + \\
&\phantom{\Delta_{11}(w)= -\sum\biggl[\sum\biggl(} + (p_1\cdots p_{j-1})a_j(1-b_ja_j^{-1}b_j^{-1})z_j\otimes (p_1\cdots p_{i-1})a_ib_ia_i^{-1}y_i\biggr) + \\
&\phantom{\Delta_{11}(w)= -\sum\biggl[} + (p_1\cdots p_{i-1})(1-a_ib_ia_i^{-1})y_i\otimes (p_1\cdots p_{i-1})a_ib_ia_i^{-1}y_i + \\
&\phantom{\Delta_{11}(w)= -\sum\biggl[} + (p_1\cdots p_{i-1})a_iz_i\otimes (p_1\cdots p_{i-1})a_ib_ia_i^{-1}y_i \biggr] + \\
&\phantom{\Delta_{11}(w)=} + \sum_{i=1}^{n}\biggl[\sum_{j=1}^{i-1}\biggl((p_1\cdots p_{j-1})(1-a_jb_ja_j^{-1})y_j\otimes (p_1\cdots p_{i-1})a_iz_i + \\
&\phantom{\Delta_{11}(w)= +\sum\biggl[\sum\biggl(} + (p_1\cdots p_{j-1})a_j(1-b_ja_j^{-1}b_j^{-1})z_j \otimes (p_1\cdots p_{i-1})a_iz_i\biggr) + \\
&\phantom{\Delta_{11}(w)= +\sum\biggl[}+(p_1\cdots p_{i-1})y_i \otimes (p_1\cdots p_{i-1})a_iz_i\biggr] - \\
&\phantom{\Delta_{11}(w)=} - \sum_{i=1}^{n-1}\biggl[\sum_{j=1}^i (p_1\cdots p_{j-1})(1-a_jb_ja_j^{-1})y_j\otimes (p_1\cdots p_i)z_i + \\
&\phantom{\Delta_{11}(w)= -\sum\biggl[\sum\biggl(} (p_1\cdots p_{j-1})a_j(1-b_ja_j^{-1}b_j^{-1})z_j\otimes (p_1\cdots p_i)z_i\biggr] - \\
&\phantom{\Delta_{11}(w)=} - z_n\otimes b_ny_n.
\end{align*}
}
\end{theorem}

\begin{proof}
There is a contracting homotopy $s$ for the resolution $P$ such that
\begin{align*}
s_{-1}(1) &= x, \\
s_0(\tilde{p}x) &= \ds\sum_{j=1}^n\left(\frac{\partial \tilde{p}}{\partial a_j}y_j + \frac{\partial \tilde{p}}{\partial b_j}z_j\right),
\end{align*}
where $\tilde{p}$ is a reduced word in $G$ (we must, of course, be
careful in our choice of $\tilde{p}$ in order to keep ourselves from
defining $s_0$ in two different ways for the same element of
$P_0$). Once we have $s_{-1}$ and $s_0$, we can quickly compute
$\Delta_0\colon P_0 \to P_0\otimes P_0$ and $\Delta_1\colon P_1 \to
(P\otimes P)_1$ using Proposition~\ref{proposition:tomodadiagonal}. We
get
\begin{align*}
\Delta_0(x) &= s_{-1}\varepsilon(x)\otimes s_{-1}\varepsilon(x) = x\otimes x,\\
\Delta_1(y_i) &= \tilde{s}_0\Delta_0d_1(y_i) = \tilde{s}_0(a_ix\otimes a_ix) - \tilde{s}_0(x\otimes x) \\
&= y_i\otimes a_ix + x\otimes y_i, \\
\Delta_1(z_i) &= \tilde{s}_0\Delta_0d_1(z_i) = \tilde{s}_0(b_ix\otimes b_ix) - \tilde{s}_0(x\otimes x) \\
&= z_i\otimes b_ix + x\otimes z_i. \\
\end{align*}
Now, to calculate the maps $\Delta_{02}\colon P_2 \to P_0\otimes P_2$
and $\Delta_{11}\colon P_2 \to P_1\otimes P_1$, we observe the
following: if $g,g'\in G$, then
\begin{align*}
\tilde{s}_1(gx\otimes g'y_i) &= \underbrace{s_0(gx)\otimes g'y_i}_{\in P_1\otimes P_1} + \underbrace{x\otimes s_1(g'y_i)}_{\in P_0\otimes P_2}, \\
\tilde{s}_1(gx\otimes g'z_i) &= \underbrace{s_0(gx)\otimes g'z_i}_{\in P_1\otimes P_1} + \underbrace{x\otimes s_1(g'z_i)}_{\in P_0\otimes P_2}, \\
\tilde{s}_1(gy_i \otimes g'x) &= s_1(gy_i)\otimes g'x \in P_2\otimes P_0,\\
\tilde{s}_1(gz_i \otimes g'x) &= s_1(gz_i)\otimes g'x \in P_2\otimes P_0.\\
\end{align*}
For the computation of $\Delta_{11}$, we are interested in the terms
that belong to $P_1\otimes P_1$, and for $\Delta_{02}$ we want the
terms belonging to $P_0\otimes P_2$. So our knowledge of $s_0$ is
sufficient for the calculation of $\Delta_{11}$. Using
Proposition~\ref{proposition:tomodadiagonal}, we get $\Delta_{11}$ as
in the statement of the theorem.

Finally, to compute $\Delta_{02}$, we need the map $s_1$ of our
contracting homotopy $s$. The homomorphism $s_1$ must satisfy
$s_1(d_2(w)) = w$, which is equivalent to
\[
s_1\left(\sum_{i=1}^n \frac{\partial p}{\partial a_i}y_i +
\frac{\partial p}{\partial b_i}z_i\right) = w.
\]
But we have 
\begin{align*}
\Delta_{02}(w) &= \pi_{02}\tilde{s}_1\Delta_1d_2(w)\\
&= \pi_{02}\tilde{s}_1\left( \frac{\partial p}{\partial a_i}(x\otimes y_i) + \frac{\partial p}{\partial b_i}(x\otimes z_i) \right) \\
&= x\otimes s_1\left(\sum_{i=1}^n \frac{\partial p}{\partial a_i}y_i + \frac{\partial p}{\partial b_i}z_i\right) = x\otimes w.
\end{align*}
\end{proof}

\begin{corollary}
Let $M$ be a trivial $\zz G$-module. If $u,v \in \Hom_{\zz G}(P_1,M)$,
then the product $[u]\smile [v]\in H^2(G,M\otimes M)$ is represented
by the map $(u\smile v)\in \Hom_{\zz G}(P_2, M\otimes M)$ defined by
\[
(u\smile v)(w) = \sum_{i=1}^n \bigl(u(y_i)\otimes v(z_i) - u(z_i)\otimes v(y_i)\bigr).
\]
\end{corollary}

\begin{proof}
We only need to notice that if $M$ is a trivial $\zz G$-module and
$\varepsilon\colon \zz G\to \zz$ is the augmentation homomorphism,
then for every $f\in \Hom_{\zz G}(P_1,M)$ we have $f(k\alpha) =
\varepsilon(k)f(\alpha)$, $\forall k\in \zz G,\: \forall \alpha\in
P_1$. This, together with the formula we have for $\Delta_{11}$, gives
us the desired result.
\end{proof}

Now we consider the case of the non-orientable surfaces: let $N_n$ be
the non-orientable surface with fundamental group given by the
presentation
\[
G = \pi_1(N_n) = \gen{a_1,\ldots,a_n \mid a_1^2a_2^2\cdot\cdots\cdot a_n^2},
\]
for $n\ge 2$. We also define $p = a_1^2a_2^2\cdot\cdots\cdot
a_n^2$. The computations for the case of a non-orientable surface are
similar to those we made in the orientable case, and so we simply
state the free resolution and the diagonal approximation we get.

\bigskip
\begin{proposition}[Free resolution for the non-orientable case]\label{proposition:nonorientableresolution}
A free resolution of $\zz$ over $\zz G$ is given by
\begin{equation}\label{eq:nonorientableresolution}
\xymatrix{
0 \ar[r] & P_2 \ar[r]^-{d_2} & P_1 \ar[r]^-{d_1} & P_0 \ar[r]^-\varepsilon & \zz \ar[r] & 0,
}
\end{equation}
where
\begin{align*}
P_0 &= \gen{x} \cong \zz G,\\
P_1 &= \gen{y_1,\ldots,y_n} \cong \zz G^n,\\
P_2 &= \gen{w} \cong \zz G,
\end{align*}
and the maps $\varepsilon$, $d_1$ and $d_2$ are defined by
\begin{align*}
\varepsilon(x) &= 1, \\
d_1(y_i) &= (a_i-1)x, \\
d_2(w) &= \ds\sum_{i=1}^n\frac{\partial p}{\partial a_i}y_i,
\end{align*}
where the partial derivatives are the Fox derivatives.
\end{proposition}

\begin{theorem}[Diagonal approximation for the non-orientable case]
Let $P$ be the free resolution of $\zz$ over $\zz G$ given by
Proposition~\ref{proposition:nonorientableresolution}. A diagonal
approximation $\Delta\colon P\to (P\otimes P)$ is partially given by
{\allowdisplaybreaks
\begin{align*}
\Delta_0\colon & P_0 \to (P\otimes P)_0 \\
&\Delta_0(x) = x\otimes x,\\
\Delta_1\colon & P_1 \to (P\otimes P)_1 \\
&\Delta_1(y_i) = y_i\otimes a_ix + x\otimes y_i, \\
\Delta_{02}\colon & P_2 \to (P_0\otimes P_2) \\
&\Delta_{02}(w) = x\otimes w, \\
\Delta_{11}\colon & P_2 \to (P_1\otimes P_1) \\
&\Delta_{11}(w) = \sum_{i=1}^n\biggl[\biggl(\sum_{j=1}^{i-1} (a_1^2\cdots a_{j-1}^2)(1+a_j)y_j\otimes (a_1^2\cdots a_{i-1}^2)y_i\biggr) + \\
&\phantom{\Delta_{11}(w) = \sum\biggl[} + \biggl(\sum_{j=1}^{i-1} (a_1^2\cdots a_{j-1}^2)(1+a_j)y_j\otimes (a_1^2\cdots a_{i-1}^2)a_iy_i\biggr) + \\
&\phantom{\Delta_{11}(w) = \sum\biggl[} + (a_1^2\cdots a_{i-1}^2)y_i \otimes (a_1^2\cdots a_{i-1}^2)a_iy_i\biggr].
\end{align*}
}
\end{theorem}

\section{Comments about applications}

Now we make some comments about applications of the results of the
previous section.  The cohomolo\-gy ring of a surface group with
twisted coefficients has been used in many applications. As we've
stated in the Introductio, see~\cite{DacibergEdson97} for an example,
where the local group is $\qq$, the rationals.

\subsection{ The Cohomology ring of a surface with arbitrary coefficients}

 In  \cite{Lyndon50}  the cohomology groups of a group which admits a presentation  with one single relation
 was studied. This includes  the surface groups for  surfaces distinct  from $S^2$ and $\rr P^2$. More precisely  
  in Corollary 11.2   they provide a formula for   $H^n(G,K)$,  $n\geq 2$. They prove the following result:
 
 \begin{corollary}(\cite{Lyndon50})
If G is defined by a single relation R, where $R=Q^q$ for no $q>1$,
and if $K$ is any left G-module, then
$H^2(G,K)=K/\bigl(\dfrac{\partial R}{\partial
  x_1},\ldots,\dfrac{\partial R}{\partial x_m}\bigr)K$ and $H^n(G,
K)=0$ for all $n > 2$.
 \end{corollary}


Using the resolution of the previous section we can write a formula
for $H^1$. For, let $G$ be a surface group with $G \ne \zz_2$, $G \ne
\{1\}$ and let $K$ be a left $G$-module. It is well known that $H^0(G,
K)=K^G$, i.e. the set of elements of $K$ which are fixed by the ring
$\zz G$. If $G$ is an orientable surface group, the group cohomology
with coefficients $K$, from section 1, is given by the homology of the
complex

\begin{equation}\label{eq:orientableresolutionK}
\xymatrix{
0 \ar[r] & K \ar[r]^-{d_1^*} & K^{2n} \ar[r]^-{d_2^*} & K  \ar[r] & 0,}
\end{equation}


\noindent where $d_1^*(k)=(a_1(k)-k, a_2(k)-k, ...., a_n(k)-k,
b_1(k)-k,..., b_n(k)-k)$ and $d_2(k_1,..., k_{2n})= \sum_{i=1}^{i=n}
\frac{\partial p}{\partial a_i}(k_i)+\frac{\partial p}{\partial
  b_i}(k_{i+n}) $. We then get the following proposition:

\begin{proposition} For $G$ an orientable surface group we have  $H^1(G, K)=Ker (d_2^*)/d_1^*(K/K^{G})$ 
where   $d_2^*$ is given as above.
\end{proposition}

\begin{remark}
 The calculation of the Fox derivative is straightforward. See, for
 example, \cite{Fadell83}.
\end{remark}

\begin{remark}
If $G=\zz\oplus\zz$ and $K=\zz\oplus\zz$, let us consider the $\zz
G$-module structure given by $\theta(1,0)=\theta(0,1)\in GL_2(\zz)$
equal to the matrix
$$
\varphi =\begin{bmatrix}
                      2 & 1 \\
                      1 & 1
         \end{bmatrix}.
$$ Since the matrix has no eigenvalue 1, follows that $H^0(G,K)=0$. A
direct calculation shows that $d_2^*$ is surjective. So it follows
that $H^2(G, K)=0$. Also by a straightforward long calculation, we
have $H^1(G, K)=0$.  This phenomena, i.e., the existence of a system
of coefficients such that the cohomology is trivial, can not happen if
the surface doesn't have Euler characteristic zero.

A similar result holds for the nonorientable surfaces and we leave the
details to the reader.
\end{remark}


\subsection{ Classification of Torus bundle over a surface}

Let us consider the $n$-dimensional torus $T^n$, which is a Lie-group.
There is a subfamily of the family of all torus bundles over a given
surface whose elements are the principal torus bundles. It is well
known that such bundles are classified by $[S, BT^n]$, (where $BT^n$
is the classifying space of the torus), which in turn is $H^2(S,
\zz\oplus\cdots\oplus\zz)=\zz\oplus\cdots\oplus\zz$. In case we
consider non-principal bundles, let us focus on those
with a prescribed action $\theta: \pi_1(S) \to
\Aut(\zz\oplus\cdots\oplus\zz)= GL_n(\zz)$. Such bundles are
classified and they are in one-to-one correspondence with
$H^2(S,\zz\oplus\cdots\oplus\zz)_{\theta}$, the second cohomology
group with local coefficients given by $\theta$. See Hilman
\cite{Hillman}, section 5, Lemma 5.1 and Robinson
\cite{Robinson72}. This group is well known and it is isomorphic to
the quotient of $\zz\oplus\cdots\oplus\zz$ by the action, i.e., by the
subgroups generated by the elements $\theta(g)(x)-x$ for all $g\in
\pi_1(S)$ and $x\in (\zz\oplus\cdots\oplus\zz)$.  It is easy to find many
actions $\theta $ such that the group $H^2(S,
\zz\oplus\cdots\oplus\zz)_{\theta}$ is trivial, i.e. there is only one
bundle with that prescribed action.

\subsection{ Cohomology of surface  bundles over a surface}

The total space $S$ of a surface bundle $S_ 1 \to S \to S_2$ over a
surface $S_2$ is a $4$-manifold. The calculation of the cohomology of
those spaces can be approached using spectral sequences. The $E_2$
term of the spectral sequence is given by $H^p(S_2, H^q(S_1))$ where
the cohomology is with local coefficients. The action of $\pi_1(S_2)$
on $H^q(S_1)$ is completely determined by the action on $H^1(S_1)$ and
this is a data of the bundle. So one can make use of the section one
to determine the ring structure of $E_2$. We point out that this is a
relevant step to find the cohomology of the space but in general not a
sufficient one.

\subsection{ The Cohomology ring of a surface with arbitrary coefficients $\zz$}

Here we  compute  the complete ring strucuture     of the cohomology of a surface for   any local 
system having as group the integers $\zz$. 
First we consider the case where $G$ is the fundamental group of an
orientable surface. If $\tilde{\zz}$ is a non-trivial $G$-module, then
it is proved in~\cite{DacibergJohn2010} that the only action
$\theta\colon G \to \Aut(\zz) = \{1,-1\}$ we need to consider is the
one given by $\theta(b_n)(1)=-1$, $\theta(a_i)(1)=1$ for all $1\le
1\le n$ and $\theta(b_i)(1)=1$ for all $1\le i\le n-1$. We denote by
$\tilde{\zz}$ the $G$-module $\zz$ given by the action $\theta$ and
denote by $\zz$ the trivial $G$-module, and we also observe that
$\zz\otimes \zz \cong \zz$, $\zz\otimes \tilde{\zz} \cong \tilde{\zz}$
and $\tilde{\zz}\otimes \tilde{\zz} \cong \zz$.

\begin{theorem} Let $G = \gen{a_1,b_1,a_2,b_2,\ldots,a_n,b_n \mid
  [a_1,b_1][a_2,b_2]\cdot\cdots\cdot [a_n,b_n]}$, for $n\ge 1$. The
  groups $H^*(G,\zz)$ and $H^*(G,\tilde{\zz})$ are given by
\begin{align*}
H^0(G,\zz) &\cong \zz, & H^0(G,\tilde{\zz}) &= 0, \\
H^1(G,\zz) &\cong \zz^{2n}, & H^1(G,\tilde{\zz}) &\cong \zz^{2n-2}\oplus \zz_2, \\
H^2(G,\zz) &\cong \zz, & H^2(G,\tilde{\zz}) &\cong \zz_2.\\
\end{align*}
More precisely, in terms of the free resolution obtained in
Proposition~\ref{proposition:orientableresolution}, we have the
following sets of generators for the groups $H^*(G,\zz)$ and
$H^*(G,\tilde{\zz})$:
\begin{align*}
\text{generator of $H^0(G,\zz)$} &: \{[x^*]\}, \\
\text{generators of $H^1(G,\zz)$} &: \{[y_1^*], \ldots, [y_n^*], [z_1^*], \ldots, [z_n^*]\}, \\
\text{generator of $H^2(G,\zz)$} &: \{[w^*]\}, \\
\phantom{blank line} \\
\text{generators of $H^1(G,\tilde{\zz})$} &: \{[y_1^*]_\theta, \ldots, [y_{n-1}^*]_\theta, [z_1^*]_\theta, \ldots, [z_n^*]_\theta\}, \quad\text{(remark: $[z_n^*]_\theta$ has order $2$)} \\
\text{generator of $H^2(G,\tilde{\zz})$} &: \{[w^*]_\theta\}, \\
\end{align*}
where $[\phantom{M}]$ and $[\phantom{M}]_\theta$ represent cohomology
classes in $H^*(G,\zz)$ and $H^*(G,\tilde{\zz})$, respectively. Finally, we have
\[
\begin{array}{l}
{[y_i^*]^2} = 0 \quad\forall 1\le i\le n,\\
{[z_i^*]^2} = 0 \quad\forall 1\le i\le n,\\
{[y_i^*]\smile[y_j^*]} = 0, \\
{[z_i^*]\smile[z_j^*]} = 0, \\
{[y_i^*]\smile[z_j^*]} = \begin{cases}
[w^*], & \text{if $i=j\in\{1,\ldots,n\}$},\\
0, & \text{if $i\ne j$},
\end{cases}\\
{[y_i^*]_\theta^2} = 0 \quad\text{for $1\le i\le n-1$}, \\
{[z_i^*]_\theta^2} = 0 \quad\text{for $1\le i\le n$}, \\
{[y_i^*]_\theta\smile[z_i^*]_\theta} = [w^*] \quad\text{for $1\le i\le n-1$}, \\
{[y_i^*]_\theta\smile[z_j^*]_\theta} = 0 \quad\text{ if $i\ne j$}, \\
{[y_i^*]_\theta\smile[y_j^*]_\theta} = 0 \quad\text{ if $i\ne j$}, \\
{[z_i^*]_\theta\smile[z_j^*]_\theta} = 0 \quad\text{ if $i\ne j$}, \\
{[y_i^*]\smile[y_j^*]_\theta} = 0 \quad\forall (i,j)\in \{1,\ldots,n\}\times \{1,\ldots,n-1\}, \\
{[z_i^*]\smile[z_j^*]_\theta} = 0 \quad\forall i,j\in \{1,\ldots,n\}, \\
{[y_i^*]\smile[z_j^*]_\theta} = \begin{cases}
[w^*]_\theta, & \text{if $i=j\in\{1,\ldots,n-1\}$},\\
0, & \text{otherwise}.
\end{cases}\\
\end{array}
\]
\end{theorem}

\begin{proof}
Follows by routine calculation using the diagonal approximation.
\end{proof}

The non-orientable case is similar. Let $G = \gen{a_1,\ldots,a_n \mid
  a_1^2a_2^2\cdot\cdots\cdot a_n^2}$ for $n\ge 2$.
In~\cite{DacibergJohn2010} it is proved that, if we want to consider
all the possible structures of $\zz$ as a $G$-module, then we only
need to look at three different actions $\theta\colon G\to \Aut(\zz)$:
the first one is the trivial action $\theta_0\colon G\to
\Aut(\zz)$. The second one is $\theta_1 \colon G\to \Aut(\zz)$,
defined by $\theta_1(a_1)(1) = -1$ and $\theta_1(a_i)(1) = 1$ for
$2\le i\le n$. The third is $\theta_2\colon G\to \Aut(\zz)$, given by
$\theta_2(a_1)(1) = \theta_2(a_2)(1) = -1$ and $\theta_2(a_i)(1) = 1$
for $3\le i\le n$. We write $\zz_{\theta_i}$ for the $G$-module $\zz$
determined by the action $\theta_i$, for $0\le i\le 2$. Observe that,
for $0\le i\le 2$, $\zz_{\theta_0}\otimes \zz_{\theta_i}
\cong\zz_{\theta_i}$, $\zz_{\theta_i}\otimes \zz_{\theta_i} \cong
\zz_{\theta_0}$, and $\zz_{\theta_1}\otimes \zz_{\theta_2}
\cong\zz_{\theta_1}$.

\begin{theorem}
$G = \gen{a_1,\ldots,a_n \mid a_1^2a_2^2\cdot\cdots\cdot a_n^2}$,
  where $n\ge 2$. The cohomology groups $H^*(G,\zz_{\theta_i})$ are
  given by
\begin{align*}
H^0(G,\zz_{\theta_0}) &\cong \zz, & H^0(G,\zz_{\theta_1}) &=0, & H^0(G,\zz_{\theta_2}) &=0, \\
H^1(G,\zz_{\theta_0}) &\cong \zz^{n-1}, & H^1(G,\zz_{\theta_1}) &\cong \zz^{n-2}\oplus\zz_2, & H^1(G,\zz_{\theta_2}) &\cong \zz^{n-2}\oplus \zz_2, \\
H^2(G,\zz_{\theta_0}) &\cong \zz_2, & H^2(G,\zz_{\theta_1}) &\cong \zz_2, & H^2(G,\zz_{\theta_2}) &\cong \zz_2.
\end{align*}
More precisely, in terms of the free resolution obtained in
Proposition~\ref{proposition:nonorientableresolution}, we can
determine explicit generators for the groups $H^*(G,\zz_{\theta_i})$
as follows:
\begin{align*}
H^0(G,\zz_{\theta_0}) &= \gen{[x^*]_{\theta_0}}, \\
H^1(G,\zz_{\theta_0}) &= \bigoplus_{k=1}^{n-1}\gen{[y_k^*-y_{k+1}^*]_{\theta_0}}, \\
H^2(G,\zz_{\theta_0}) &= \gen{[w^*]_{\theta_0}}, \\
\phantom{0} \\
H^0(G,\zz_{\theta_1}) &= 0, \\
H^1(G,\zz_{\theta_1}) &= \gen{[y_1^*]_{\theta_1}}\oplus\left(\bigoplus_{k=2}^{n-1}\gen{[y_k^*-y_{k+1}^*]_{\theta_1}}\right), \\
H^2(G,\zz_{\theta_1}) &= \gen{[w^*]_{\theta_1}}, \\
\phantom{0} \\
H^0(G,\zz_{\theta_2}) &= 0, \\
H^1(G,\zz_{\theta_2}) &= \gen{[y_1^*]_{\theta_2}}\oplus\gen{[y_1^*+y_2^*]_{\theta_2}}\oplus\left(\bigoplus_{k=3}^{n-1}\gen{[y_k^*-y_{k+1}^*]_{\theta_2}}\right), \\
H^2(G,\zz_{\theta_2}) &= \gen{[w^*]_{\theta_2}},
\end{align*}
where $[\phantom{M}]_{\theta_i}$ represents the cohomology class in
$H^*(G,\zz_{\theta_i})$ and the elements $[y_1^*]_{\theta_1}$,
$[y_1^*+y_2^*]_{\theta_2}$ and $[w^*]_{\theta_i}$ have order
$2$. Finally, the products
\[
H^1(G,\zz_{\theta_i})\otimes H^1(G,\zz_{\theta_j}) \stackrel{\smile}{\to} H^2(G,\zz_{\theta_i}\otimes\zz_{\theta_j}) 
\]
are given by
{\allowdisplaybreaks
\begin{align*}
& [y_k^*-y_{k+1}^*]_{\theta_0}^2 = 0, \\
&[y_k^*-y_{k+1}^*]_{\theta_0} \smile [y_\ell^*-y_{\ell+1}^*]_{\theta_0} =
\begin{cases}
[w^*]_{\theta_0}, & \text{if $\ell=k\pm 1$}, \\
0, & \text{otherwise},
\end{cases} \\
&[y_k^*-y_{k+1}^*]_{\theta_0} \smile [y_1^*]_{\theta_1} =
\begin{cases}
[w^*]_{\theta_1}, & \text{if $k = 1$}, \\
0, & \text{if $k>1$},
\end{cases} \\
&[y_k^*-y_{k+1}^*]_{\theta_0} \smile [y_\ell^*-y_{\ell+1}^*]_{\theta_1} =
\begin{cases}
[w^*]_{\theta_1}, & \text{if $\ell=k\pm 1$}, \\
0, & \text{otherwise},
\end{cases} \\
&[y_k^*-y_{k+1}^*]_{\theta_0} \smile [y_\ell^*-y_{\ell+1}^*]_{\theta_2} =
\begin{cases}
[w^*]_{\theta_2}, & \text{if $\ell=k\pm 1$}, \\
0, & \text{otherwise},
\end{cases} \\
&[y_k^*-y_{k+1}^*]_{\theta_0} \smile [y_1^*]_{\theta_2} =
\begin{cases}
[w^*]_{\theta_2}, & \text{if $k = 1$}, \\
0, & \text{otherwise},
\end{cases} \\
&[y_k^*-y_{k+1}^*]_{\theta_0} \smile [y_1^*+y_2^*]_{\theta_2} =
\begin{cases}
[w^*]_{\theta_2}, & \text{if $k = 2$}, \\
0, & \text{otherwise},
\end{cases} \\
&[y_1^*]_{\theta_1}^2=[w^*]_{\theta_0}, \\
&[y_k^*-y_{k+1}^*]_{\theta_1}^2 = 0, \\
&[y_1^*]_{\theta_1}\smile[y_k^*-y_{k+1}^*]_{\theta_1} = 0, \\
&[y_1^*]_{\theta_2}^2 = [w^*]_{\theta_0}, \\
&[y_1^*+y_2^*]_{\theta_2}^2 = 0, \\
&[y_1^*]_{\theta_2}\smile [y_1^*+y_2^*]_{\theta_2} = [w^*]_{\theta_0}, \\
&[y_1^*]_{\theta_2}\smile [y_k^*-y_{k+1}^*]_{\theta_2} = 0, \\
&[y_1^*+y_2^*]_{\theta_2}\smile [y_k^*-y_{k+1}^*]_{\theta_2} = 0, \\
&[y_k^*-y_{k+1}^*]_{\theta_0} \smile [y_\ell^*-y_{\ell+1}^*]_{\theta_1} =
\begin{cases}
[w^*]_{\theta_1}, & \text{if $\ell=k\pm 1$}, \\
0, & \text{otherwise},
\end{cases} \\
&[y_1^*]_{\theta_1}\smile [y_1^*]_{\theta_2} = [w^*]_{\theta_1}, \\
&[y_1^*]_{\theta_1}\smile [y_1^*+y_2^*]_{\theta_2} = 0, \\
&[y_1^*]_{\theta_1}\smile [y_\ell^*-y_{\ell+1}^*]_{\theta_2} = 0, \\
&[y_k^*-y_{k+1}^*]_{\theta_1}\smile [y_1^*]_{\theta_2} = 0, \\
&[y_k^*-y_{k+1}^*]_{\theta_1}\smile [y_1^*+y_2^*]_{\theta_2}
= \begin{cases}
[w^*]_{\theta_1}, & \text{if $k=2$},\\
0, & \text{otherwise},
\end{cases} \\
&[y_k^*-y_{k+1}^*]_{\theta_1}\smile [y_\ell^*+y_{\ell+1}^*]_{\theta_2}
= \begin{cases}
[w^*]_{\theta_1}, & \text{if $\ell= k \pm 1$},\\
0, & \text{otherwise}.
\end{cases}
\end{align*}
}
\end{theorem}
\begin{proof}
Follows by routine calculation using the diagonal approximation.
\end{proof}

\begin{remark}
The cohomology ring structure with arbitrary coefficients of the
groups of the form $G\rtimes \zz_2$, where $G$ is a surface group,
seems an interesting problem.  More precisely, those groups are
natural groups to act freely and properly on even dimensional homotopy
spheres. Hence certain cohomological properties of those groups are
expected to show up. We hope to pursue this idea somewhere.
\end{remark}

\bibliographystyle{ws-ijac}
\bibliography{surfacesdiagonal-2014-04-02}

\begin{thebibliography}{10}

\bibitem{LuciliaDaciberg2008}
L.~D. Borsari and D.~L. Gon{\c c}alves, The first group (co)homology of a group
  ${G}$ with coefficients in some ${G}$-modules, {\em Quaest. Math.} {\bf {\bf
  31}}  (2008)  89--100.

\bibitem{Brown82}
K.~S. Brown, {\em Cohomology of groups} (Springer, New York, 1982).

\bibitem{Fadell83}
E.~Fadell and S.~Husseini, The {N}ielsen number on surfaces, {\em Topol.
  Methods Nonlinear Funct. Anal.} {\bf {\bf 21}}  (1983)  59--98.

\bibitem{Fox63}
R.~H. Fox and R.~H. Crowell, {\em Introduction to Knot Theory} (Ginn and Co.,
  Boston, Mass., 1963).

\bibitem{DacibergEdson97}
D.~L. Gon{\c c}alves and E.~Oliveira, The {L}efschetz coincidence number for
  maps among compact surfaces, {\em Far East J. Math. Sci. (FJMS)} {\bf {\bf
  2}}  (1997)  147--166.

\bibitem{DacibergJohn2010}
D.~L. Gonçalves and J.~Guaschi, The {B}orsuk-{U}lam theorem for maps into a
  surface, {\em Topology Appl.} {\bf {\bf 157}}  (2010)  1742--1759.

\bibitem{Hillman}
J.~A. Hillman, {\em Four-manifolds, geometries and knots} (University of
  Warwick, Mathematics Institute: Geometry and Topology Publications, Sidney,
  2002).

\bibitem{Lyndon50}
R.~C. Lyndon, Cohomology theory of groups with a single defining relation, {\em
  Ann. of Math. (2)} {\bf {\bf 52}}(3)  (1950)  650--665.

\bibitem{Robinson72}
C.~A. Robinson, Moore-{P}ostnikov systems for non-simple fibrations, {\em
  Illinois J. Math.} {\bf {\bf 16}}(2)  (1972)  234--242.

\bibitem{TomodaPeter2008}
S.~Tomoda and P.~Zvengrowski, Remarks on the cohomology of finite fundamental
  groups of 3-manifolds, {\em Geom. Topol. Monogr.} {\bf {\bf 14}}  (2008)
  519--556.

\bibitem{Wall61}
C.~T.~C. Wall, Resolutions for extensions of groups, {\em Math. Proc. Cambridge
  Philos. Soc.} {\bf {\bf 57}}  (1961)  251--255.

\end{thebibliography}
\nocite{*}

\end{document}